\documentclass[11pt]{article}
\usepackage[T1]{fontenc}
\usepackage[font=small,labelfont=bf,tableposition=top]{caption}

\usepackage[psamsfonts]{amssymb}
\usepackage{amsmath,amsfonts}
\usepackage{amsthm}
\usepackage{graphicx}
\usepackage{bbm}
\usepackage{mathrsfs}

\usepackage[T1]{fontenc}

\DeclareCaptionLabelFormat{andtable}{#1~#2  \&  \tablename~\thetable}

\newcommand{\num}{\text{Num}}
\newcommand{\numC}{\text{\emph{Num}}}
\newcommand{\AAA}{a_{\max}}

\newcommand{\T}{\mathcal{T}}

\newtheorem{Theorem}{Theorem}
\newtheorem*{Theorem*}{Theorem}

\newtheorem{lemma}{Lemma}
\newtheorem{proposition}{Proposition}

\theoremstyle{definition}
\newtheorem{definition}{Definition}

\theoremstyle{remark}

\begin{document}
\title{Limiting curves for   polynomial adic systems.\footnote{Supported by the RFBR (grant 14-01-00373)}}

\author{A.~R.~Minabutdinov\thanks{National Research University Higher School of Economics, Department of Applied Mathematics and Business Informatics, St.Petersburg, Russia, e-mail: \texttt{ aminabutdinov@hse.ru.}}}
\maketitle
\begin{abstract}
We prove the existence and describe limiting curves resulting from deviations in partial sums in the ergodic theorem for  cylindrical functions and   polynomial  adic systems. For a general ergodic measure-preserving transformation and a summable function we give a necessary condition for  a limiting curve to exist. Our work generalizes results by \'E. Janvresse, T. de la Rue and Y. Velenik and answers several questions from their work.

\end{abstract}
{\bf Key words:} Polynomial adic systems, ergodic theorem, deviations in ergodic theorem.

\emph{MSC:}   37A30, 28A80

\section{Introduction}

In this paper we develop the notion of a limiting curve introduced  by    \'E.~Janvresse, T. de la Rue and Y. Velenik in \cite{DeLaRue}. Limiting curves were studied for   the Pascal adic in \cite{DeLaRue} and \cite{LodMin15}. In this paper we study it  for  a wider class of adic transformations.

Let  $T$ be a measure preserving transformation defined on a Lebesgue probability space $(X, \mathcal{B},\mu)$ with an invariant ergodic probability measure $\mu$. Let    $g$ denote a function in $L^1(X,\mu)$. Following \cite{DeLaRue} for a point $x\in X$ and a positive integer $j$ we denote the partial sum $\sum\limits_{k=0}^{j-1}g\big(T^kx\big)$ by $S_{x}^g(j)$. We extend the function $S_{x}^g(j)$ to a real valued argument by a linear interpolation and denote extended function by  $F_{x}^g(j)$ or simply $F(j),j\geq 0$.

Let  $(l_n)_{n=1}^{\infty}$ be a sequence of positive integers. We consider continuous on $[0,1]$ functions $\varphi_n(t) = \frac{F(t\cdot l_n(x)) - t \cdot F(l_n)}{R_{n}} \big( \equiv\varphi_{x,l_n}^g(t)  \big),$ where the normalizing coefficient $R_{n}$ is canonically defined to be equal to the maximum in $t\in[0,1]$ of $|F(t\cdot l_n(x)) - t \cdot F(l_n)|$.

\begin{definition}\normalfont 
If there is a sequence $l^g_n(x)\in\mathbb{N}$  such that functions $\varphi_{x,l^g_n(x)}^g$ converge to a (continuous) function $\varphi_x^g$ in sup-metric on $[0,1],$ then
the graph of the limiting function $\varphi=\varphi^g_x$ is called a \emph{limiting curve}, sequence $l_n=l_n^g(x)$ is called a \emph{stabilizing sequence} and the sequence $R_{n} = R_{x,l^g_n(x)}^g$ is called a \emph{normalizing sequence}.  The quadruple $\Big(x, \big(l_n\big)_{n=1}^\infty, \big(R_{n}\big)_{n=1}^\infty, \varphi\Big)$ is called a \emph{limiting bridge}.
\label{def:LimitShape}
\end{definition}

Heuristically, the limiting curve describes   small fluctuations (of certainly renormalized) ergodic sums $\frac{1}{l}F(l), l\in(l_n),$ along the forward trajectory $x, T(x), T^2(x)\dots$. More specifically, for $l\in(l_n)$ it holds $F(t\cdot l) = t F(l)+R_l \varphi(t)+o(R_l)$, where $t\in[0,1].$

In this paper we will always assume that $T$ is an \emph{adic transformation}. Adic transformations were introduced into ergodic theory by A.~M.~Vershik in ~\cite{Ver81} and were extensively  studied since that time.  The following important theorem shows that adicity assumption is not restrictive at all:

\begin{Theorem*}\emph{(A.~M.~Vershik, \cite{ver82})}.
\label{Thm:VershikSuperTheorem}
Any ergodic measure preserving transformation  on a Lebesgue space  is  isomorphic to some adic transformation. Moreover, one can find such an isomorphism that any given countable dense invariant subalgebra of measurable sets goes over into the algebra of cylinder sets.
\end{Theorem*}

In \cite{ver82,VerLiv92,Ver14} authors encouraged studying different approaches to combinatorics of Markov's compacts (sets of paths in Bratteli diagrams).
In particular, it is interesting to find a natural class of adic transformations such that  the limiting bridges exist for cylindric functions. Moreover, it is interesting to study joint growth rates of stabilizing and normalizing sequences.

In this paper we give necessary condition for a limiting curve to exist.
Next we find necessary and sufficient conditions  for almost sure (in $x$) existence of limiting curves  for a class
of self-similar adic transformations and cylindric functions.  These transformations (in a slightly less generality) were considered by   X. Mela and S. Bailey in \cite{Mela2006} and \cite{Bailey2006}. Our work 	extends  \cite{DeLaRue} and answers several questions from this research.

\section{Limiting curves and cohomologous to a constant functions}

In this section we show that a necessary condition for  limiting curves to exist is unbounded growth of the normalizing coefficient $R_n$. Contrariwise we show that normalizing coefficients are bounded if and only if function $g$ is cohomologous to a constant. In particular this implies that there are no limiting curves for cylindric functions for an ordinary odometer.

\subsection{Notions and definitions}

Let $ B=B(\mathcal{V},\mathcal{E})$ denote a Bratteli diagram defined by the set of vertices $\mathcal{V}$ and the set of edges $\mathcal{E}$. Vertices at the level $n$ are numbered $k=0$ through $L(n)$. We associate to a Bratteli diagram $B$  the space $X = X(B)$ of infinite edge paths  beginning at the vertex $v_0 = (0, 0).$ Following fundamental paper~\cite{Ver81} we assume that there is a linear order  $\leq_{n,k}$ defined on the set of edges with a terminate vertex $(n,k),  0\leq k\leq L(n)$. These linear orders define a lexicographical order on the set of edges paths in $X$ that belong to the same class of the tail partition.  We denote by   $\preceq$ corresponding partial order on $X$. The set of maximal (minimal) paths is defined by $X_{\max}$ (correspondingly, $X_{\min}$).

\begin{definition} \normalfont Adic transformation  $T$ is defined on  $X\setminus\big(X_{\max}\cup X_{\min}\big)$ by setting $Tx, \, x\in X,$ equal to the successor  of $x$, that is, the smallest $y$ that satisfies $y \succ x$.
\label{def:adicSystems}
\end{definition}

Let  $\omega$ be a path in $X$. We denote by $(n,k_n(\omega))$ a vertex through which $\omega$ passes at level $n$.
For a finite path $c=(c_1,\dots,c_n)$ we denote   $k_n(c)$ simply by $k(c)$. A cylinder set $C=[c_1c_2\dots c_{n}]=\{\omega\in X|\omega_{1}=c_1,\omega_{2}=c_2,\dots, \omega_{n}=c_{n} \}$ of a rank $n$ is totally defined by a finite path from the vertex $(0,0) $  to the vertex $(n,k)=(n,k(c)).$ Sets  $\pi_{n,k}$ of lexicographically ordered finite paths $c=(c_0,c_1,\dots,c_{n-1})$, $k(c)=n$, are  in one to one correspondence with  towers $\tau_{n,k}$ made up of corresponding cylinder sets $C_j = \tau_{n,k}(j),1\leqslant j\leqslant\dim(n,k)$.
The  dimension $\dim(n,k)$ of the vertex $(n,k)$ is the total number of such finite paths (rungs of the tower).

We denote by  $\num(c)$ the number  of  finite paths in lexicographically ordered set $\pi_{n,k}$. Evidently, $1\leq \num(c)\leq \dim(n,k).$  For a given level $n$ the set of towers $\{\tau_{n,k}\}_{0\leq k\leq L(n)}$ defines approximation of transformation $T$, see~\cite{Ver81}, \cite{VerLiv92}.

We can consider a vertex $(n,k) $ of Bratteli diagram $B$ as an origin  in a new diagram  $B'_{n,k}=(\mathcal{V}',\mathcal{E}')$. The set of vertices $\mathcal{V}' $, edges $\mathcal{E}'$  and edges paths $X(B'_{n,k})$ are naturally defined. As above  partial order  $\preceq'$    on $X(B')$ is induced by linear orders $\leq_{n',k'}, n'>n$.

\begin{definition}\label{def:delfSimilar} Ordered Bratteli diagram  $(B,\preceq)$ is  \emph{self-similar} if ordered diagrams $(B,\preceq)$ and $(B'_{n,k},\preceq'_{n,k})$ are isomorphic $n\in \mathbb{N},$ $0\leqslant k\leqslant L(n).$
\label{def:selSimilarBratteli}
\end{definition}

Let $\mathscr{F}$ denote the set of all functions $f:X\rightarrow \mathbb{R}.$
We denote by  $\mathscr{F}_N$ the space of cylindric functions of rank $N$ (i.e. functions that are constant on cylinders of rank  $N$).

Let $g\in\mathscr{F}_N, \  N< n.$ We denote by $F_{n,k}^g $ linearly interpolated partial sums $S^g_{x\in\tau_n^k(1)}$. Assume that self-similar Bratteli  diagram $B$ has  $L+1$ vertices at level $N$ and let $\omega\in\pi_{n,k}, 0\leqslant k \leqslant L(n),$ be a finite path such that its initial segment $\omega'=(\omega_{1},\omega_{2},\dots, \omega_{N})$ is a maximal path, i.e. $\num(\omega')=\dim(N,k(\omega'))$. Let ${E}^{N,l}_{n,k}$ denote the number of paths from $(0,0)$ to $(n,k)$ passing through the vertex $(N,l),0\leq l\leq L,$ and not exceeding path $\omega$. We denote by ${\partial}^{N,l}_{n,k}(\omega)$ the ratio of  ${E}^{N,l}_{n,k}$ to $\dim(N,l).$ It is not hard to see
that a partial sum $F_{n,k}^g$ evaluated at   $j=\num(\omega)$ has the following expression:

\begin{equation}
\label{eq:sumOfCylFunction}
F^g_{n,k}(j) = \sum\limits_{l=0}^{L} h^g_{N,l}{\partial}^{N,l}_{n,k}(\omega),\end{equation}
where coefficients $h^g_{N,l}$ are equal to  $F_{N,l}^g(H_{N,l}),\, 0\leqslant l\leqslant L$.

Expression \eqref{eq:sumOfCylFunction} is a generalization of Vandermonde's convolution formula.

\subsection{A necessary condition for existence of limiting curves}

Let $(X,T)$ be an ergodic measure-preserving transformation with invariant measure $\mu$. Let  $g$ be a  summable function and a point  $x\in X$.  We consider a sequence of functions  $\varphi^g_{x,l_n} $ and normalizing coefficients $R_{x,l_n}^g$ given  by the identity \[\varphi_{x,l_n(x)}^g(t) = \frac{S_{x}^g([t\cdot l_n(x)]) - t \cdot S_{x}^g(l_n(x))}{R_{x,l_n(x)}^g},\] where $R_{x,l_n(x)}^g$ equals maximum of absolute value of the numerator. Without loss of generality, we assume that the limit $g^*(x)=\lim\limits_{n\rightarrow\infty}\frac{1}{n}S^g_x $ exists at the point $x$. The following theorem generalizes Lemma $2.1$ from \cite{DeLaRue} for an arbitrary summable function.

\begin{Theorem}
If a continuous limiting curve $\varphi_x^g=\lim_n\varphi^g_{x,l_n} $ exists for $\mu$-a.e. $x$, then the normalizing coefficients  $R_{x,l_n}^g$ are unbounded in $n$.
\label{Th:LimitingCurveNesCond}
\end{Theorem}
\begin{proof}
Assume the contrary that  $|R_{x,l_n}^g|\leqslant K$. For simplicity we introduce the following notation:  $S=S_x^g,$ $\varphi_n = \varphi^g_{x,l_n},$ $R_n = R_{x,l_n}^g$ and  $ \varphi = \varphi_x$. Since $\varphi\neq0,$ there is $j\in\mathbb{N}$ such that  $\frac{1}{j}S(j)\neq g^*.$ This in turn implies $\liminf_n\big|\varphi_n(\frac{j}{l_n})\big|= \liminf_n \frac{1}{R_n}\big|S(j)-\frac{j S(l_n)}{l_n}\big|\geqslant \frac{1}{K}|S(j)-jg^*|=\frac{j}{K}\big|\frac{1}{j}S(i)-g^*\big|>0,$
contradicting continuity of the limiting curve $\varphi$ at the origin.

\end{proof}

\begin{definition}\normalfont  A function $g\in L^\infty(X,\mu)$  ($\mu$-a.e.) of the form $g~=~c+h\circ T - h $ for some $c\in \mathbb{R}$ and $h\in L^\infty(X,\mu)$ is called cohomologous to a constant in $L^\infty$.
\end{definition}

\begin{Theorem}
\label{Th:RBoundnessIsEquiToCohomolToConst}
Normalizing sequence $R_{x,l_n}^g$ is bounded if and only if function $g$ is cohomologous to a constant.
\end{Theorem}
\begin{proof}
Sums $\sum\limits_{j=0}^{n-1} (g-g^*)\circ T^j$ of a cohomologous function are $\mu$-a.e. bounded, therefore normalizing  coefficients~$R_{x,l_n}^g$ are $\mu$-a.e. bounded too.

The proof of the converse statement exploits the result by A.~G. Kachurovskiy from \cite{Kac96}. Assume that the normalizing coefficients $R_{x,l_n}^g$ are bounded. Then for $\mu$-a.e. point $x\in X$ and for any $j\in\mathbb{N}$ the following inequality holds $|S^g_x(j)-\frac{j}{l_n}S^g_x(l_n)|\leqslant C$. Going to the limit in $n$, we see that
$|\sum\limits_{i=1}^{j}f\circ T^i(x)|\leqslant C,$
where $f = g-g^*$. Theorem $19$  from \cite{Kac96} (see also  G.~Halasz, \cite{Halasz1976}), inequality $|S^f_x|\leqslant C,$ is equivalent to existence of a function $h\in L^\infty$, such that $f = h\circ T-h.$ Therefore $g$ equals to $h\circ T-h+g^*.$
\end{proof}

\begin{definition}\normalfont
Let $B$ be a Bratteli diagram such that there is only one vertex at each
level, and let the edge ordering be such that the edges increase from left to right.  This transformation is called an \emph{odometer}. A stationary odometer
is an odometer for which the number of edges connecting consecutive levels is constant.
\end{definition}

\begin{figure}[h]
                \centering
                \includegraphics[scale=0.7]{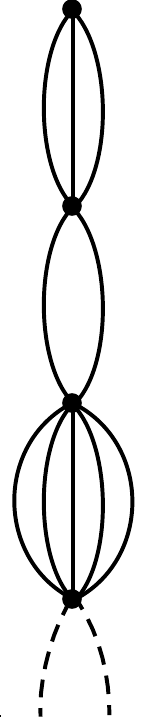}
                \caption{A Bratteli diagram of an odometer.}
\end{figure}

\begin{Theorem} Let $(X,T)$ be an odometer. Any cylindric function $g\in\mathscr{F}_N$  is cohomologous to a constant. Therefore there is no limiting curve for a cylindric function.
\end{Theorem}

\begin{proof} There is only one vertex  $(n,0)$ at each level $n.$ Expression~ \eqref{eq:sumOfCylFunction} for the partial sum $F_{n,0}^g(i)$ is  evidently  valid for any odometer (even without assumption of self-similarity).  Moreover, expression  \eqref{eq:sumOfCylFunction}  is defined by the only coefficient  $h^g_{N,0}$ and therefore is proportional to $H_N = \dim(N,0)$. We can subtract such  constant  $C$ to the function $g$
that equality  $h^{g-C}_{N,0}=0 $ holds. But this is  equivalent to the following: The function function $g-C$ belongs to the linear space spanned by the functions $f_{j} - f_j\circ T, 1\leqslant j\leqslant H_N,$ where $f_{j}$  is the indicator-function of the $j$-th rung in the tower $\tau_{N,0}$. Therefore function $g-C$ is cohomologous to zero.
\end{proof}

\section{Existence of limiting curves for polynomial adic systems}

In this part we will show that any not cohomologous to a constant cylindric function in a polynomial adic system has a limiting curve. These generalizes Theorem 2.4. from \cite{DeLaRue}.

\subsection{Polynomial adic systems}

Let $p(x) = a_0 + a_1 x\dots + a_{d} x^{d} $  be an integer polynomial of degree   $d\in \mathbb{N}$ with positive integer coefficients $a_i, 0\leq i\leq d.$ Bratteli diagram $B_{p} =(\mathcal{V},\mathcal{E})_{p} $ associated to  polynomial  $p(x)$ is defined as follows:
\begin{enumerate}
\item Number of vertices grows linearly: $|\mathcal{V}_0|=1$ и $|\mathcal{V}_n| = |\mathcal{V}_{n-1}|+d=nd+1, n\in \mathbb{N}$.
\item If $0\leqslant j\leqslant d$ vertices $(n,k)$ and $(n+1,k+j)$ are connected by   $a_j$ edges. 
\end{enumerate}
Polynomial $p(x)$ is called a \emph{generating polynomial} of the diagram~$B_{p} $, see paper \cite{Bailey2006} by S.~Bailey.

\begin{figure}[h]
                \centering
                \includegraphics[scale=1]{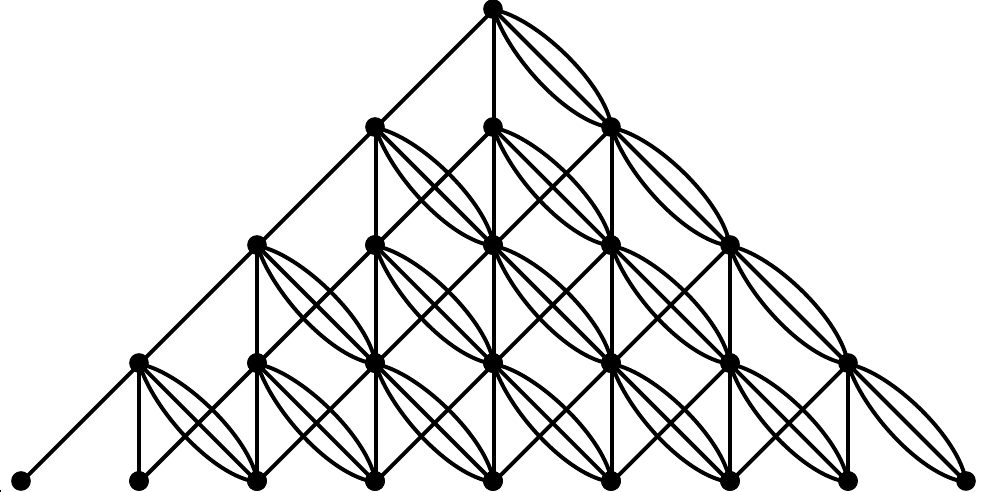}
                \caption{Bratteli diagram associated to polynomial $1+x+3x^2.$}
\end{figure}

Since the number  of edges into vertex $(n, k)$
is exactly $p(1)=a_0+a_1+\dots+a_d$ it is natural to use the alphabet $\mathcal{A} = \{0,1,\dots,a_0+a_1+\dots+a_d-1\}$ for edges labeling. We call  a lexicographical order  defined in \cite{Bailey2006} \emph{a canonical order}. It is defined  as follows: Edges connecting $(0, 0)$ with
$(1, d)$ are labeled through  $0$ to  $a_d-1$ (from left to right); edges connecting $(0,0)$ and $(1,d-1)$, are indexed by $a_d$ to $a_d+a_{d-1}-1$,  etc. Edges connecting  $(0,0)$ and $(1,0)$, are indexed through $a_0+a_1+\dots+a_{d-1}$ to $a_0+a_1+\dots+a_{d}$.

Infinite paths are totally defined by this labeling and may be considered as one sided infinite sequences in  $\mathcal{A}^{\mathbb{N}}.$ We denote the path space by  $X_p.$

\begin{figure}[h]
                \centering
                \includegraphics[scale=1.2]{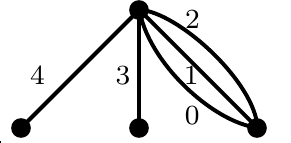}
                \caption{Labeling of the polynomial  system associated to $1+x+3x^2.$}
                \label{fig:Labeling}
\end{figure}

We denote by $T_{p}$ the adic transformation associated with the canonical ordering.

\textbf{Remark}. Any self-similar Bratteli diagram is either a diagram of a stationary odometer or is associated to some polynomial $p(x)$. Any non-canonical ordering is obtained from canonical by some substitution $\sigma$.

Everywhere below we stick to the  canonical order. Case of general order needs several straightforward changes that are left to the reader.

Dimension of the vertex $(n,k)$ from diagram $B_{p}$ equals to  the coefficient of $x^k$ in the polynomial $(p(x))^n$ and is called \emph{ generalized binomial coefficient}. We denote it by $C_p(n,k)$. For $n>1$ coefficients $C_p(n,k)$ can be evaluated by a  recursive expression $C_p(n,k)=\sum_{j=0}^da_jC_d(n-1,k-j)$.

In  \cite{Mela2006} and \cite{Bailey2006} X.~ M\'ela and S.~Bailey showed that the fully supported invariant ergodic measures of the system $(X_{p},T_{p})$ are the one-parameter family of Bernoulli measures:

\begin{Theorem}\emph{(S.~Bailey, \cite{Bailey2006}, X.~ M\'ela, \cite{Mela2006})} 1. Let $q\in(0,\frac{1}{a_0})$ and $t_q$ is the unique solution in $(0,1)$  to the equation
\[a_0q^d+a_1q^{d-1}t+\dots+a_dt^d-q^{d-1}=0,\]
then the invariant, fully supported, ergodic probability measures for the adic transformation $T_{p}$ are the one-parameter family of Bernoulli measures $\mu_q, {q\in(0,\frac{1}{a_0})},$
\[\mu_q=\prod_{0}^\infty\bigg(\underbrace{q,\dots,q}_{a_0},\underbrace{t_q,\dots,t_q}_{a_1},\underbrace{\frac{t_q^2}{q},\dots,\frac{t_q^2}{q}}_{a_2},\dots,\underbrace{\frac{t_q^d}{q^{d-1}},\dots,\frac{t_q^d}{q^{d-1}}}_{a_d}\bigg).\]  2.Invariant measures that
are not fully supported are\[\prod_{0}^\infty\bigg(\underbrace{\frac{1}{a_0},\dots,\frac{1}{a_0}}_{a_0},0,\dots,0\bigg) \quad \text{and} \quad \prod_{0}^\infty\bigg(0,\dots,0,\underbrace{\frac{1}{a_d},\dots,\frac{1}{a_d}}_{a_d},\bigg).\]
\label{Th:SLPmeasures}
\end{Theorem}
\begin{definition} \emph{Polynomial adic system associated with polynomial} $p(x)$, is a triple $(X_{p},T_{p},\mu_q),\ q\in(0,\frac{1}{a_0}).$
\label{def:SLPsystem}
\end{definition}

In particular, if $p(x)=1+x$ system $(X_{p},T_{p},\mu_q),\ q\in(0,1),$ is the well-known Pascal adic transformation. Transformation was defined  in \cite{ver82} by A.~M.~Vershik  \footnote{However earlier isomorphic transformation was used by \cite{HajanItoKakutani} and \cite{Kakutani1976}.} and was studied in many works \cite{MelaPetersen, PascalLooselyBernoulli, Vershik11, Ver15}, see more complete list in the last two papers. For the Pascal adic space $X_p$ is an infinite dimensional unit cube $I=\{0,1\}^\infty$, while measures   $\mu_q$  are dyadic Bernoulli measures $\prod_1^\infty(q,1-q)$. Transformation $T_p=P$ is defined by the following formula (see \cite{ver82})\footnote{$P^k(x), \ k\in\mathbb{Z},$ is defined for all $x$  except eventually diagonal, i.e.,  except those $x$ for which there exists $ n \in \mathbb{N}$ such that either $x_k = 0$ for all $k \geq n$ or $x_k = 1$ for all $k \geq n$}:

\begin{equation}
\label{eq:VershikDef}
 x\mapsto Px; \ \ P(0^{m-l}1^l\textbf{10}\dots)=1^l0^{m-l}\textbf{01}\dots\end{equation}
 (that is only first $m+2$ coordinates of $x$ are being changed). De-Finetti's theorem and Hewitt-Savage $0$--$1$ law imply that all $P$-invariant ergodic measures are the Bernoulli measures  $\mu_p = \prod_1^\infty(p,1-p)$, where $0<p<1$.

Below we enlist several known properties of the polynomial systems:
\begin{enumerate}
\item Polynomial systems are weakly bernoulli  (the proof  essentially  follows \cite{PascalLooselyBernoulli} and is performed in \cite{Mela2006} and \cite{Bailey2006}).
\item Complexity function has polynomial growth rate (for the Pascal adic first term of asymptotic expansion is known to be equal to $\frac{n^3}{6}$, see  \cite{MelaPetersen}).
\item Polynomial system $(X_{p},T_{p},\mu_q)$ defined by a polynomial $p(x) = a_0+a_1x$ with  $a_0 a_1>1$ has a non-empty set of non-constant eigenfunctions.
\end{enumerate}

Authors of  \cite{DeLaRue} studied limiting curves  for the Pascal adic transformation $(I,P,\mu_q), q\in(0,1)$.

\begin{Theorem}
 $($\emph{\cite{DeLaRue}}, Theorem $2.4.)$\label{Th:ExistanceOfLimitShapePN} Let  $P$ be the Pascal adic transformation defined on Lebesgue  probability space $(I, \mathcal{B},\mu_{q}),q\in(0,1),$  and $g$  be a cylindric function from $\mathscr{F}_N$. Then for $\mu_q$-a.e. $x$
limiting curve $\varphi^g_x \in C[0,1]$ exists if and only if  $g$ is not cohomologous to a constant.
\end{Theorem}

For the Pascal adic limiting curves can be  described by nowhere differentiable functions, that generalizes Takagi curve.

\begin{Theorem} $($\emph{\cite{LodMin15}}, Theorem $1.)$ Let  $P$ be the Pascal adic transformation defined on the Lebesgue space  $(I, \mathcal{B},\mu_{q})$, $N\in\mathbb{N}$ and $g\in \mathscr{F}_N$ be a not cohomologous to a constant cylindric function. Then for  $\mu_{{q}}$-a.e. $x$ there is a stabilizing sequence $l_n(x)$ such that  the limiting function is $ \alpha_{g,x} \mathcal{T}^1_{q}$, where $ \alpha_{g,x}\in\{-1,1\} $, and  $\mathcal{T}^1_{q}$ is given by the identity \[\mathcal{T}^1_{q} (x) = \frac{\partial F_{\mu_q}}{\partial q}\circ F_{\mu_q}^{-1}(x),\,  x\in[0,1],\]
where $F_{\mu_q}$ is the distribution function\footnote{More precisely $F_{\mu_q}$ is distribution function of measure $\tilde{\mu}_q$, that is image of $\mu_q$ under canonical mapping $\phi:I\rightarrow[0,1]$, $\phi(x)=~\sum\limits_{i=1}^\infty\frac{x_i}{2^i}.$ } of ${\mu_q}.$
\label{Thm:TakagiOnly}
\end{Theorem}

The graph of   $\frac12\mathcal{T}^1_{1/2}$ is the famous Takagi curve, see \cite{Takagi1903}.

For a function $g$ correlated with the  indicator functions of $i$-th  coordinate $\mathbbm{1}_{\{x_{i}=0\}}$,  $x=(x_j)_{j=1}^\infty\in X$ Theorem \ref{Thm:TakagiOnly} was proved in \cite{DeLaRue}.

\subsection{Combinatorics of finite paths in the polynomial adic systems}

In this section we'll specify representation \eqref{eq:sumOfCylFunction} for the  polynomial adic systems.

For a finite path $\omega=(\omega_1,\omega_2,\dots, \omega_n)$ we set  $k^1(\omega)$ equal to $nd-k(\omega).$ Using self-similarity of the diagram $B_p$ we can inductively prove the following explicit expression for  $\num(\omega)$:

\begin{proposition}
Index $\numC(\omega)$ of a finite path $\omega=(\omega_j)_{j=1}^n$ in lexicographically ordered set  $\pi_{n,k(\omega)}$ is defined by equality:
\begin{equation}\label{eq:Numformula}\numC(\omega)  = \sum_{j=2}^r \sum\limits_{i=0}^{\omega_{a_j}-1}C_P(a_j-1;k^1(\omega)-k^1(i)-m_j) +\numC(\omega_1),\end{equation}
where $m_j =\sum\limits_{t=j}^{r-1}k^1(\omega_{a_t}),2\leqslant j\leqslant r-1,$ $m_r=0$, and polynomial $P(x)$ is given by the identity $P(x) = x^dp(x^{-1})$.
\end{proposition}

\textbf{Remark}.
If initial segment $(\omega_1,\omega_2,\dots,\omega_N)$ of $\omega\in \pi_{n,k}$ is a maximal path to some vertex $(N,l)$, then \eqref{eq:Numformula} can be rewritten as follows:
\begin{equation}\label{eq:Numformula2}\text{Num}(\omega)  = \sum_{j=N+1}^r \sum\limits_{i=0}^{\omega_{a_j}-1}C_P(a_j-1;k^1(\omega)-k^1(i)-m_j) +C_P(a_N;k^1(\omega)-m_l) .\end{equation}

Let $N$ and $l,\, 0\leqslant l\leqslant Nd,$ be positive integers and $\omega\in\pi_{n,k}$ be a finite path. Function $\partial_{k^1}^{N,l}: \mathbb{Z}_+\rightarrow\mathbb{Z}_+ , k^1=nd-k,$ is defined by the identity

\begin{equation}\label{eq:PartialGeneralSLP}\partial_{k^1}^{N,l}\omega =  \sum_{j=N+1}^r \sum\limits_{i=0}^{\omega_{a_j}-1}C_P(a_j-1-N;k^1-k^1(i)-m_j-l),\end{equation}
where positive integers  $a_j,k^1(i),m_j,$ are defined as in~\eqref{eq:Numformula2}.
Parameters $N$  and $l$ correspond to shifting the origin vertex $(0,0)$ to the vertex $(N,l)$. Therefore value of the function  $\partial_{k^1}^{N,l}\omega,k^1=k^1(\omega),$  equals to the number of paths from the vertex $(0,0)$ going through the vertex $(N,l)$ to the vertex  $(n,k),k=nd-k^1$, and non-exceding path $\omega$ divided by $\dim(N,l)$.

Let $K_{M,n,k}$, $1\leqslant M\leqslant n$,   denote indexes (in lexicographical order) of those paths $\omega=(\omega_1,\dots, \omega_n)\in\pi_{n,k}$, such that their initial segment $(\omega_1,\omega_2,\dots,\omega_M)$ is maximal (as a path from $(0,0)$ to some vertex $(M,l)$).

Let  $g\in \mathscr{F}_N$. Function $\tilde{F}^{g,M}_{n,k}: K_{M,n,k}\rightarrow \mathbb{R}$ (where $M, N\leqslant M\leqslant n$ is a positive integer) is defined by the identity
\begin{equation}
\label{eq:PartialSumFormula}
\tilde{F}_{n,k}^{g,M}(j) =\sum\limits_{l=0}^{Nd}h_{M,l}^g  \partial_{nd-k}^{M,l}\, \omega,\end{equation}
where $\num(\omega)=j$, $j\in K_{M,n,k}$, $\omega\in \pi_{n,k}$.
We  extend domain of the function $\tilde{F}_{n,k}^{g,M}$  to the whole interval $[1,H_{n,k}]$ using linear interpolation. Expression~\eqref{eq:sumOfCylFunction} implies that for  $j\in K_{M,n,k}$ the identity $\tilde{F}^{g,M}_{n,k}(j)=F^g_{n,k}(j)$ holds. Non strictly speaking,
higher values of parameter $M ,M>N,$ makes functions $\tilde{F}_{n,k}^{g,M}$ to be more and more rough approximation of function $F^g_{n,k}$ and points from $K_{M,n,k}$ correspond to nodes of this approximation.

\begin{lemma}Let  $1\leqslant j\leqslant H_{n,k}$ and $g\in \mathscr{F}_N$. There exists a constant $C=C(g)$, such that the following inequality holds for all $n,k$: \[|\tilde{F}^{g,N}_{n,k}(j)-F_{n,k}^g(j)|\leqslant C.\]
\label{Lm:tildeF}
\end{lemma}

\textbf{Remark.} If the function $g$ equals $1$ and $\num(\omega)=\dim(n,k)\equiv C_P(n,k^1(\omega)),$ then expression \eqref{eq:PartialSumFormula} (as well as \eqref{eq:sumOfCylFunction}) reduces to:
\[C_P(n,k) = \sum\limits_{l=0}^{Nd}C_P(N,l)C_P(n-N,k-l), \]
that is Vandermonde's convolution formula for generalized binomial coefficients.

\subsection{A generalized  $r$-adic number system on $[0,1]$}

Let parameter $q\in(0,1/a_0)$ and number $t_q\in(0,1/a_1)$ be defined as in Theorem \ref{Th:SLPmeasures}. We denote by $r=p(1)$ number of letters in the alphabet $\mathcal{A}$.

Let  $\omega=(\omega_i)_{i=1}^\infty\in X_p,\omega_i\in \mathcal{A},$ be an infinite path. It is also natural to consider $\omega $ as a path in  an infinite  perfectly balanced tree $\mathcal{M}_r.$

By $\bar{s}_n = (s_{n}^{0},\dots{s_{n}^{r-1}})^T$ we denote  $r$-dimensional vector with $j$-th, $0\leqslant j\leqslant r-1$, component equal to number of occurrences of letter  $j$ among $(\omega_1,\omega_2,\dots,\omega_n).$ Let $\bar{a}_i,0\leqslant i\leqslant r-1,$ denote $r$-dimensional vector \[(\underbrace{0,0,\dots,0}_{\sum_{j=0}^{i-1}a_j},\underbrace{1,1,\dots,1}_{a_i},\underbrace{0,0,\dots,0}_{\sum_{j=i+1}^{d}a_j}),\]


Let $u\cdot v$ denote scalar product of $r$-dimensional vectors $u$ and $v$. We define mapping $\theta_q: X \rightarrow [0,1]$ by the following identity:
\begin{equation}\label{eq:generalizednadic(q)}x = \sum\limits_{j=1}^\infty I_q(\omega_j)q^j\Big(\frac{t_q}{q}\Big)^{\bar{a}_1\cdot\bar{s}_j+2\bar{a}_2\cdot\bar{s}_j+\dots+d\,\bar{a}_d\cdot\bar{s}_j},\end{equation}
where $I_q(w) = a_0\frac{q^{h+1}}{t_q^{h+1}}+a_1\frac{t_qq^h}{t_q^{h+1}}+\dots+a_h\frac{q}{t_q}+s$ with $w=\sum\limits_{i=0}^ha_i+s, 0\leqslant s<a_{h+1},0\leqslant h<d$.

Let  $X_0$ denote the set of stationary paths.
Function $\theta_{1/r}$  is a canonical bijection $\phi=\theta_{1/r}: X\setminus X_0\rightarrow[0,1]\setminus G,\, G = \phi(X_0).$  Function $\phi$ maps measure $\mu_q,q\in(0,1),$ defined on $X$ to measure $\tilde{\mu}_q$ on  $[0,1]$, the family of towers $\{\tau_{n,k}\}_{k=0}^{nd}$ to the family $\{\tilde{\tau}_{n,k}\}_{k=0}^{nd}$ of disjunctive intervals. That defines isomorphic realization  $\tilde{T}_p$ on  $[0,1]\setminus G$ of polynomial adic transformation $T_p$.

As shown by A.~M.~Vershik any adic transformation has a cutting and stacking realization on the subset of a full measure of $[0,1]$ interval. However, nice explicit expression   \eqref{eq:generalizednadic(q)} needs some regularity from the Bratteli diagram.

Conversely, any point  $x\in[0,1]$ could be represented by series \eqref{eq:generalizednadic(q)}. We call this representation $q$-$r$-adic representation associated to the polynomial $p(x)$. (If  $r=2$  representation \eqref{eq:generalizednadic(q)} for  $q=1/2$ is a usual dyadic representation of $x\in(0,1)$.) Let $G_q^m$ denote the set (vector) of all stationary numbers of rang $m, m\in\mathbb{N}$, i.e. numbers with a finite representation
\[x = \sum\limits_{j=1}^m I(\omega_j)\prod\limits_{i=0}^{r-1}p_i^{s_j^i},\]
and let  $G_q=\cup_m G_q^m $ be the set of all $q$-$r$-stationary numbers.

Let $l\in\mathbb{N}$ and $x$ be a path in $X_p.$ We consider  $r^l$-dimensional vectors $\tilde{K}_n={K}_{n-l,n,k_n(x)}$   and renormalization mappings $D_{n,k}:[1, C_p(n,k)]\rightarrow[0,1]$ defined by $D_{n,k}(j)=\frac{j}{C_p(n,k)}.$ Using ergodic theorem it is straightforward to show that for $\mu_q$-a.e. $x$  it holds $\lim\limits_{n\rightarrow\infty}R_{n,k_n(x)}(\tilde{K}_n)=G^l_q$ (where convergence is the componentwise convergence of vectors).

\subsection{Existence of limiting curves for polynomial adic systems}

In this part we generalize Theorem $2.4$ from \cite{DeLaRue} for polynomial adic systems~$(X_p,T_p)$ associated with  positive integer polynomial $p$.

First we prove a combinatorial variant of the theorem. Let, as above, $x\in X_p$ be an infinite path going through vertices  $(n,k_n(x))\in B_p$. Below we write vertex $(n,k_n(x))$ as $(n,k_n)$ or simply as $(n,k).$ To simplify notation, the dimension  $\dim(n,k)=C_p(n,k)$ is  denoted  by $H_{n,k}.$

We define function  ${\varphi}_{n,k}^g=\varphi_{x\in\tau_{n,k}(1),H_{n,k}}^g:[0,1]\rightarrow \mathbb{R}$ by identity
\[\varphi_{n,k}^g(t) = \frac{{F}_{n,k}^g(tH_{n,k})-t{F}_{n,k}^g(H_{n,k})}{R_{n,k}^g}.\]
Let $F$ be a function defined on $[1,H_{n,k}]$. Define function $\psi_{F,n,k}$ on $[0,1]$ by
\[\psi_{F}(t) = \frac{{F}(tH_{n,k})-t{F}(H_{n,k})}{R_{n,k}},\]
where  $R_{n,k}$ is a canonically defined normalization coefficient. Then the following identity holds $\psi_{F_{n,k}^g}=\varphi_{n,k}^g. $

Let  $g\in \mathscr{F}_N$ be not cohomologous to a constant cylindric function.
Theorem \ref{Th:RBoundnessIsEquiToCohomolToConst} implies that normalization sequence $\big(R_{n,k_n}^g\big)_{n\geqslant1}$ monotonically increases. Lemma \ref{Lm:tildeF} shows that
\[\big|\big| \psi_{F_{n,k}^g} - \psi_{F_{n,k}^{g,N}}\big|\big|_\infty\xrightarrow[n\rightarrow\infty]{}0.\]

We want to show that there is a sequence $(n_j)_{j\geqslant 1}$ and a continuous function  $\varphi(t),t\in[0,1],$  such that
\[\lim_{j\rightarrow\infty}\big|\big| \psi_{F_{n_j,k_{n_j}}^{g,N}} - \varphi\big|\big|_\infty=0.\]

Following \cite{DeLaRue}, we consider an auxiliary object: a family of polygonal functions  $\psi^M_n = \psi_{F_{n,k}^{g,n-M+N}},$ $N+1\leqslant M\leqslant n. $ Graph of each function $\psi^M_n$ is defined by $(2r)^M$-dimensional array $(x^M_i(n),y^M_i(n))_{i=1}^{r^M}$, such that  $\psi^M_n(x^M_i(n))=y^M_i(n)$. Results from Section $3.3$ show that vector $(x^M_i(n))_{i=1}^{r^M}$ converges pointwise to $q$-$r$-stationary numbers $G^M_q$ of rank $M,$ given by polynomial $p(x).$

Let $l$ and $M$ be positive integers, such that $N+1\leqslant l<M<n$. Functions $F_{n,k}^{g,n-M}$ and $F_{n,k}^{g,n-l}$ coincide at each point from  ${K}_{n-l,n,k_n}$, therefore functions  $\psi^M_n$ and $\psi^l_n$ also coincide at $(x^l_i(n))_{i=1}^{r^l}.$ Moreover, Proposition \ref{Pr:aka3.1Prop} (it generalizes Proposition $3.1$ from \cite{DeLaRue}) provides the following estimate:
\[\big|\big| \psi^M_{n_j} -\psi^l_{n_j} \big|\big|_\infty\leqslant C_1e^{-C_2(M-l)},\]
with $C_1,C_2>0.$ For a fixed  $M$ we can extract a subsequence $(n_j)$  such that polygonal functions $\psi^M_{n_j}$ converge to a polygonal function $\varphi^M$ in sup-metric. Then, as in \cite{DeLaRue},  using a standard diagonalization  procedure we can find subsequence (that again will be denoted by $(n_j)_j$) such that convergence to some continuous on $[0,1]$ function holds for any $M$: \[\lim\limits_{M\rightarrow\infty} \limsup\limits_{j\rightarrow\infty}\big|\big| \psi_{n_j}^{M} - \varphi\big|\big|_\infty=0.\]
Auxiliary functions  $\varphi^M$ are polygonal approximations to the function $\varphi.$

Therefore we have proved the following claim, generalizing Theorem \ref{Th:ExistanceOfLimitShapePN}:
\begin{Theorem}
\label{Th:ExistanceOfLimitShapeSLP} Let   $(X,T,\mu_q ),q\in(0,\frac{1}{a_0}),$  be a polynomial adic transformation defined on Lebesgue  probability space $(I, \mathcal{B},\mu_{q}),q\in(0,1),$  and $g$  be a not cohomologous to a constant cylindric function from $\mathscr{F}_N$. Then for $\mu_q$-a.e. $x$ passing through vertices $(n,k_n(x))$ we can extract a subsequence $(n_j)$ such that  $\varphi_{n_j,k_{n_j}(x)}^g$  converges in $\sup$-metric to a continuous function on $[0,1]$.
\end{Theorem}

Each limiting curve $\varphi$ is a limit in  $j$ of polygonal curves $\psi_{n_j}^m,m\geqslant 1,$ with nodes at stationary points  $G_q\subset[0,1]$. Therefore, its values
$\varphi(t)$ can be obtained as limits $ \lim\limits_{j\rightarrow\infty}\psi_{n_j}^m\big(\frac{\num(\omega)}{H_{n_j,k_{n_j}}}\big),$ where $t\in G^{m}_q$ and $\lim\limits_{j\rightarrow\infty}\frac{\num(\omega)}{H_{n_j,k_{n_j}}} = t,$ with $\num(\omega)\in K_{n_j-m,n_j,k_{n_j}}.$

Self-similar structure of towers simplifies this task. We write simply $n$ for $n_j(x)$, $F$ for $F_{n,k}^g$ and $R_n= R_{n,k}^g$.  The  following lemma in fact generalizes results from Section 3.1. of \cite{DeLaRue}:

\begin{lemma}Limiting curve is totally defined by the following limits $n\rightarrow\infty$:
\[\lim_{n\rightarrow\infty }\frac{1}{R_n}\Big(F(L_{m,i,n,k}) - \frac{L_{m,i,n,k}}{H_{n,k}}F(H_{n,k})\Big),\]
where  $L_{m,i,n,k}=\sum \limits_{j=0}^{m}a_jH_{n-i,k(\omega)+j-di},0\leqslant m\leqslant d.$
\label{Lm:triangArray}
\end{lemma}
\begin{proof} We may assume that  $\delta<\frac{k_n}{n(x)}<\delta d$ for some $\delta>0$.
First, we suppose that some typical   $n=n(x)>>m$ and $k_n$ are taken.
We consider a set of ingoing finite paths of length $m$ to the vertex $(n,k), d\leq k\leq d(n-1), n>>m$. Self-similarity of $B_p$ implies that these paths can be considered as paths going from the origin to some vertex $(m,j), 0\leq j\leq md,$ of $B_p$, see Fig. \ref{Fig:TriangularArray}. As shown in Section~$3.3$ above, each such path correspond to a point from $G_q^m,$ which in its turn correspond to $q$-$r$-adic interval of rank $m$. Let  $x_{m,j}, 0\leqslant j\leqslant md$ denote length of such interval and $y_{m,j}$,  denote increment of the function $\psi_{n}^M$ on the interval  $(m,j)$. Values $ (x_{m,j}, y_{m,j})$ may be defined inductively: For  $m=0$ by $x_{0,0}=1, y_{0,0}=0$, and for $m>0 $  and indices  $j$ such that $(m-1)d<j\leqslant md$ by $x_{m,j}=\sum\limits_{i=0}^j\frac{ a_iC_p(n-m,nd-k+j-di)}{H_{n,k}},\, y_{m,j} =\psi_{n}^M(x_{m,j}) $; for other values of $j$ by recursive expression $x_{m-1,i}=\sum_{j=0}^d a_jx_{m,i-j},$ $y_{m-1,i}=\sum_{j=0}^d a_jy_{m,i-j}.$ Therefore function $\psi_{n_j}^M$ is totally defined by its values at $x_{m,j}$, $1\leqslant m\leqslant M,$ $(m-1)d<j\leqslant md$. Going to the limit we obtain the claim.
 \end{proof}

\begin{figure}[h]
                \centering
                \includegraphics[scale=0.3]{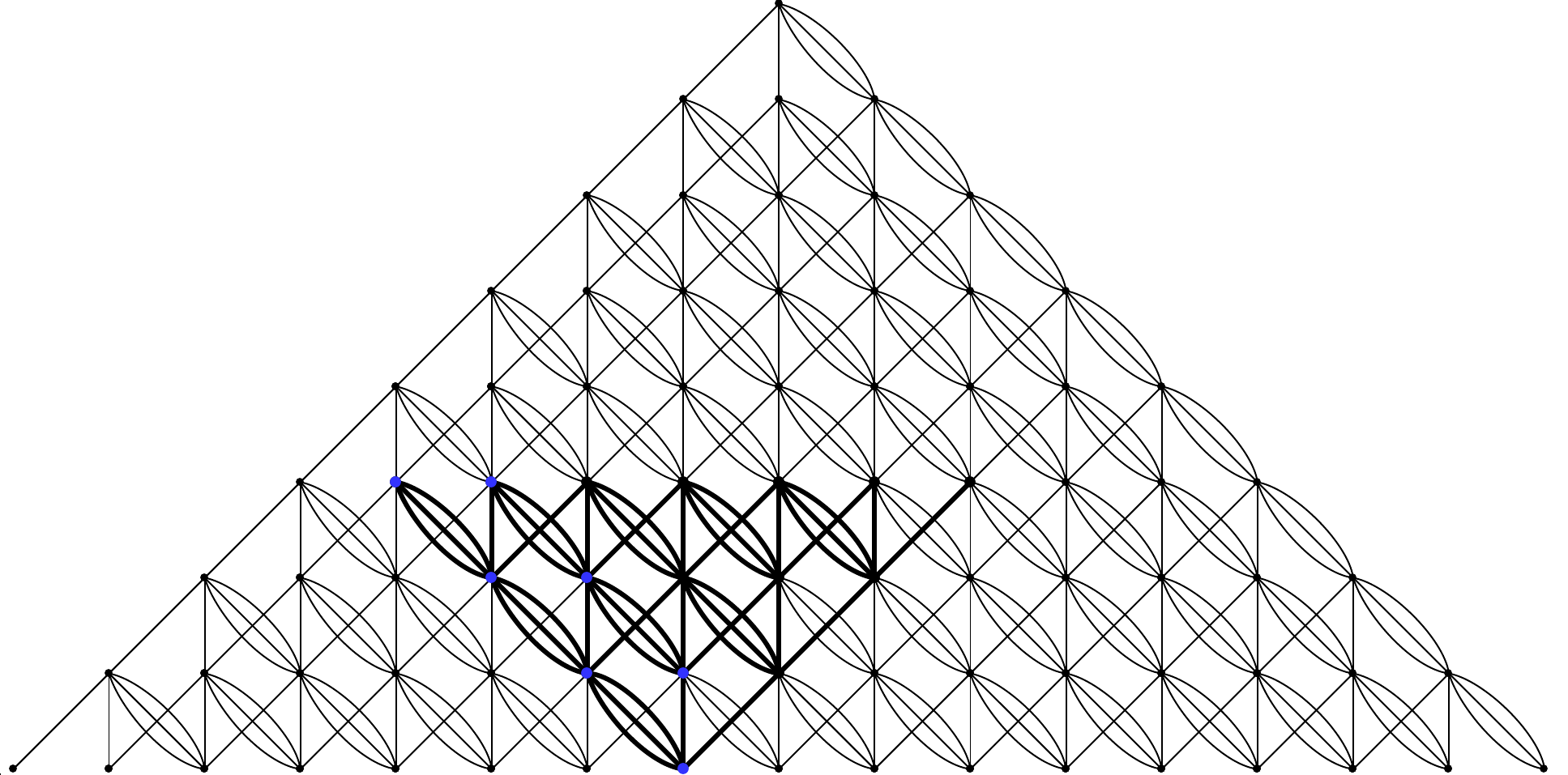}
                \includegraphics[scale=0.6]{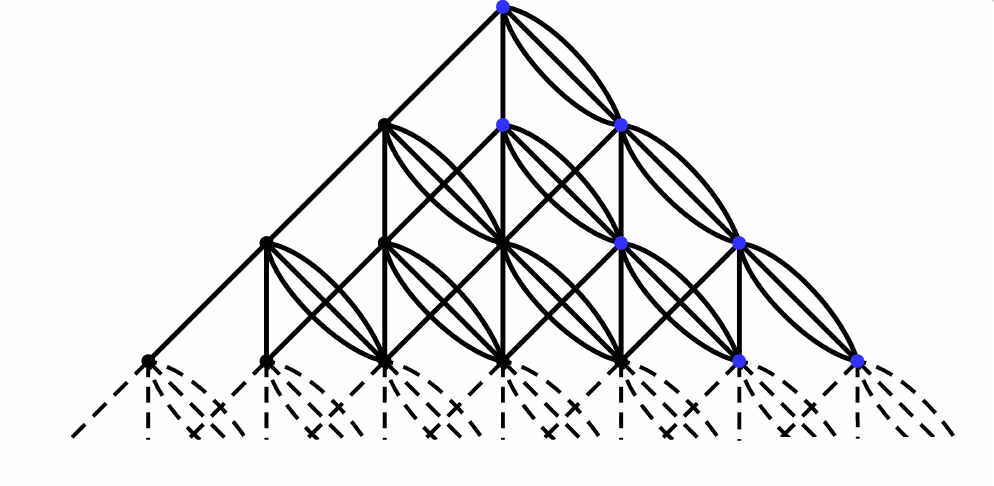}
                \caption{For Bratteli diagram $B_p$ and graph of ingoing paths. }
                \label{Fig:TriangularArray}
\end{figure}

Stochastic version of Theorem \ref{Th:ExistanceOfLimitShapeSLP} is obtained from the following claim: for any $\varepsilon>0$ for $\mu_q$-a.e. $x$ there exists subsequence  $n_j(x)$ such that $\num(w^j), w^j=(x_1,\dots x_{n_j}),$ satisfies the following condition $\frac{\num(w^j)}{H_{n_j,k_{n_j}}}<\varepsilon.$ In fact, even more strong result holds. It follows from the recurrence property of one-dimensional random walk and was first proved by \'E. Janvresse and T. de la Rue in \cite{PascalLooselyBernoulli} to show that the Pascal adic transformation is loosely Bernoulli. Later it was  generalized in
 \cite{Mela2006, Bailey2006} for the polynomial adic systems.
\begin{lemma}
\label{Th:firstFloor}
For any $\varepsilon>0$ and $\mu_q\times \mu_q$-a.e. pair of paths $(x,y)\in X\times X$ there is a subsequence $n_j$ such that $k_{n_j}(x)=k_{n_j}(y)$ and indices $\numC(\omega_x)$, $\numC(\omega_y)$ of paths $\omega_x=(x_1,x_2,\dots x_{n_j})$ and $\omega_y=(y_1,y_2,\dots y_{n_j})$ satisfy the following inequlity $\numC(\omega_z)/H_{n_j,k_{n_j}(x)}<~\varepsilon,$ $z\in\{x,y\},$ for each~$j\in\mathbb{N}$.
\end{lemma}

\begin{Theorem}\emph{(}Stochastic variant of Theorem $\ref{Th:ExistanceOfLimitShapeSLP}$.\emph{)}
Let  $(X,T,\mu_q ),$ $q\in(0,\frac{1}{a_0}),$   and $g$  be a cylindric function from $\mathscr{F}_N$. Then for $\mu_q$-a.e. $x$
limiting curve $\varphi^g_x \in C[0,1]$ exists if and only if function $g$ is not cohomologous to a constant.
\label{Th:SLP-LimitingCurveExistance}
\end{Theorem}
\begin{proof} Follows from Lemma \ref{Th:firstFloor}, Theorem \ref{Th:ExistanceOfLimitShapeSLP} and Theorem \ref{Th:RBoundnessIsEquiToCohomolToConst}.
\end{proof}

\textbf{Remark}
Lemma \ref{Th:firstFloor} implies that appropriate choice of stabilizing sequence $ l_n(x)$ can provide the same limiting curve $\varphi_x^g,$ $\lim\limits_{j\rightarrow\infty }||\varphi_{x,l_j(x)}^g-\varphi_x^g||=0$,  for $\mu_q$-a.e.~$x$.

Finally we prove Proposition \ref{Pr:aka3.1Prop} used above. It generalizes Proposition $3.1$ from \cite{DeLaRue}. However, its proof needs an additional statement due to the non unimodality of generalized binomial coefficients $C_{p}(n,k)$:

\begin{lemma} Let  $p(x)=a_0+a_1x+\dots+a_dx^d$ be a positive integer polynomial. Then the following holds:

 $1.$ There exist $n_1\in\mathbb{N}$ and  $C_1>0$, depending only on $\{a_0,\dots,a_d\}$, such that  $\max\limits_k\{\frac{C_p(n,k+1)}{C_p(n,k)},\frac{C_p(n,k)}{C_p(n,k+1)} \}\leqslant  C_1n$ for $n>n_1.$

 $2.$ $C_p(n-1, k-i)\leqslant \frac{1}{a_i}\max\{\frac{k}{n}, 1-\frac{k}{n}\}C_p(n,k),\, 0\leqslant i \leqslant d.$
 \label{Lm:genBinEstimates}
\end{lemma}
\begin{proof}
$1.$
Let  $X$ be a discrete random variable on $\{0,1,\dots, d\}$ with distribution associated to the polynomial $p(x),$ that is $\text{Prob}(X=k) =a_k/{p(1)}, 0\leqslant k\leqslant d.$ Distribution of a sum $Y_n=X_1+X_2+\dots X_n$ of i.i.d. random  variables $X_k,0\leqslant k\leqslant n,$ with distributions associated to the polynomial $p(x)$, is associated to the polynomial $p^n(x),$ i.e. $\text{Prob}(Y_n=k) =C_p(n,k)/{p^n(1)}, 0\leqslant k\leqslant nd.$ A.~Oldyzko and L.~Richmond showed in \cite{OldyzkoRichmond1985} that the function  $f_n(k)\equiv \text{Prob}(Y_n=k)$ is asymptotically unimodal, i.e. for $n\geq n_1$,   coefficients $C_p(n,k),$ $0\leqslant k \leqslant nd,n\geqslant n_1,$ first increase  (in $k$) and decrease then.

We denote by $C$ and $c$ the maximum and the minimum values of the coefficients  $\{C_p(n_1,k)\}_{k=0}^{n_1d}$ of the polynomial $p^{n_1}(x)$.  Let also $\AAA$ denote the maximum of the coefficients  $\{a_0,\dots,a_d\}$ of the polynomial $p(x)$.
We will use induction in $n$ to prove that  $\frac{C_p(n,k+1)}{C_p(n,k)}\leqslant  \AAA\frac{C}{c}dn, 0\leqslant k\leqslant nd-1, n\geqslant n_1$. (The second estimate $\frac{C_p(n,k)}{C_p(n,k+1)} \leqslant   \AAA\frac{C}{c}dn$ can be  proved in the same way).
We start now with the base case: For $n=n_1$ it obviously holds that $\frac{C_p(n_1, k+1)}{C_p(n_1, k)}\leqslant \frac{C}{c}\leqslant \frac{Cd\AAA}{c}, 0\leqslant k\leqslant n,$ hence
we have shown the base case.

Now assume that we have already shown  $\frac{C_p(n-1, k)}{C_p(n-1, k-1)}\leqslant \frac{Cd\AAA}{c}(n-1)$, where  $1\leqslant k\leqslant d (n-1)$  and $n\geq n_1.$ We need to show that $\frac{C_p(n, k)}{C_p(n, k-1)}\leqslant \frac{Cd\AAA}{c}n, $  $1\leqslant k\leqslant dn.$
\begin{multline}
\frac{C_p(n, k+1)}{C_p(n, k)} = \frac{\sum_{i=0}^d a_iC_p(n-1,k+1-i)}{ \sum_{i=0}^d a_iC_p(n-1,k-i)}
\leqslant  \\
\leqslant \frac{C_p(n-1,k)\big(a_0+a_1+\dots+a_{d-1}+d \AAA a_d\frac{C}{c}(n-1)\big)}{a_dC_p(n-1, k)}\leqslant \\
\leqslant\frac{a_0+a_1+\dots+a_{d-1}-d\AAA}{a_d}+\frac{\AAA C d n}{c}\leqslant \frac{\AAA Cd}{c}  n.
\end{multline}
$2.$ The  statement follows directly from the following identity for the generalized binomial coefficients: \begin{equation}\label{eq:GenBinIdentity}\sum\limits_{i=1}^dC_p(n-1,k-i)a_ii=\frac{k}{n}C_p(n,k).\end{equation}
To show it we differentiate identity $p^n(x)=\sum_{k\geq0} C_p(n,k) x^k$ resulting
$np^{n-1}(x)p'(x) = \sum_{k\geq0} kC_p(n,k) x^{k-1},$ $ p'(x)=a_1+2a_2x+\dots+da_dx^{d-1}.$
It remains to equate exponents from the two sides.
\end{proof}

The following proposition generalizes Propositon 3.1 from \cite{DeLaRue}. We  preserved the original notation where it was possible.

\begin{proposition} Let $N\geqslant 1$ be a positive integer and $\delta\in(0,\frac14)$ be a small parameter. Let   ${A}=A(\bar{n},\bar{k})\in B_p$ be a vertex with coordinates $\big(\bar{n},\bar{k}\big)$ satisfying
$2\delta\bar{n}\leqslant k\leqslant(d-2\delta)\bar{n}$ and
$2\delta\bar{n}\leqslant nd-k\leqslant(d-2\delta)\bar{n}$.
Let $\alpha_{l},0\leqslant l\leqslant Nd,$ be real numbers, such that $\sum\limits_{l=0}^{Nd}\alpha_{l}^2>0$. Let $n, N\leqslant n\leqslant \bar{n},$ and  ${B}(n,k)=\big(n,k\big)$ be a vertex with coordinates satisfying $0\leqslant k\leqslant \bar{k}, 0\leqslant n-k\leqslant \bar{n}-\bar{k}.$  Define
\begin{equation}\label{ineq:generalizedBinomialIneq}\gamma_{n,k} = \frac{1}{R}\sum\limits_{l=0 }^{Nd}\alpha_{l}\, C_d(n-N,k-l),\end{equation}
where $R = R(A,B,\delta)$ is a renormalization constant such that $|\gamma_{n,k}|$ are uniformly in $n$ and $k$  from $0\leqslant k\leqslant \bar{k},0\leqslant n-k\leqslant\bar{n}- \bar{k}, N\leqslant n\leqslant \bar{n}$ bounded by $2$. Then there exist a constant  $C = C(\delta,N),$ such that, provided $\bar{n}$ is large enough, the following inequality holds for all $n,k$:
\[|\gamma_{n,k}|\leqslant 3e^{-C(\bar{n}-n)}.\]
\label{Pr:aka3.1Prop}
\end{proposition}

Conditions on the vertex ${A}$ $\delta$-separate it from "boundary"\, vertices $(\bar{n},0)$ and $(\bar{n},d\bar{n})$. Conditions on the vertex $B={B}(n,k)$ provides it can be considered as a vertex in a "flipped"\, graph and that it can be connected with the vertex   $A$, see Fig.~\ref{Fig:InvertedGraph}.

\begin{figure}[h]
                \centering
                \includegraphics[scale=0.6]{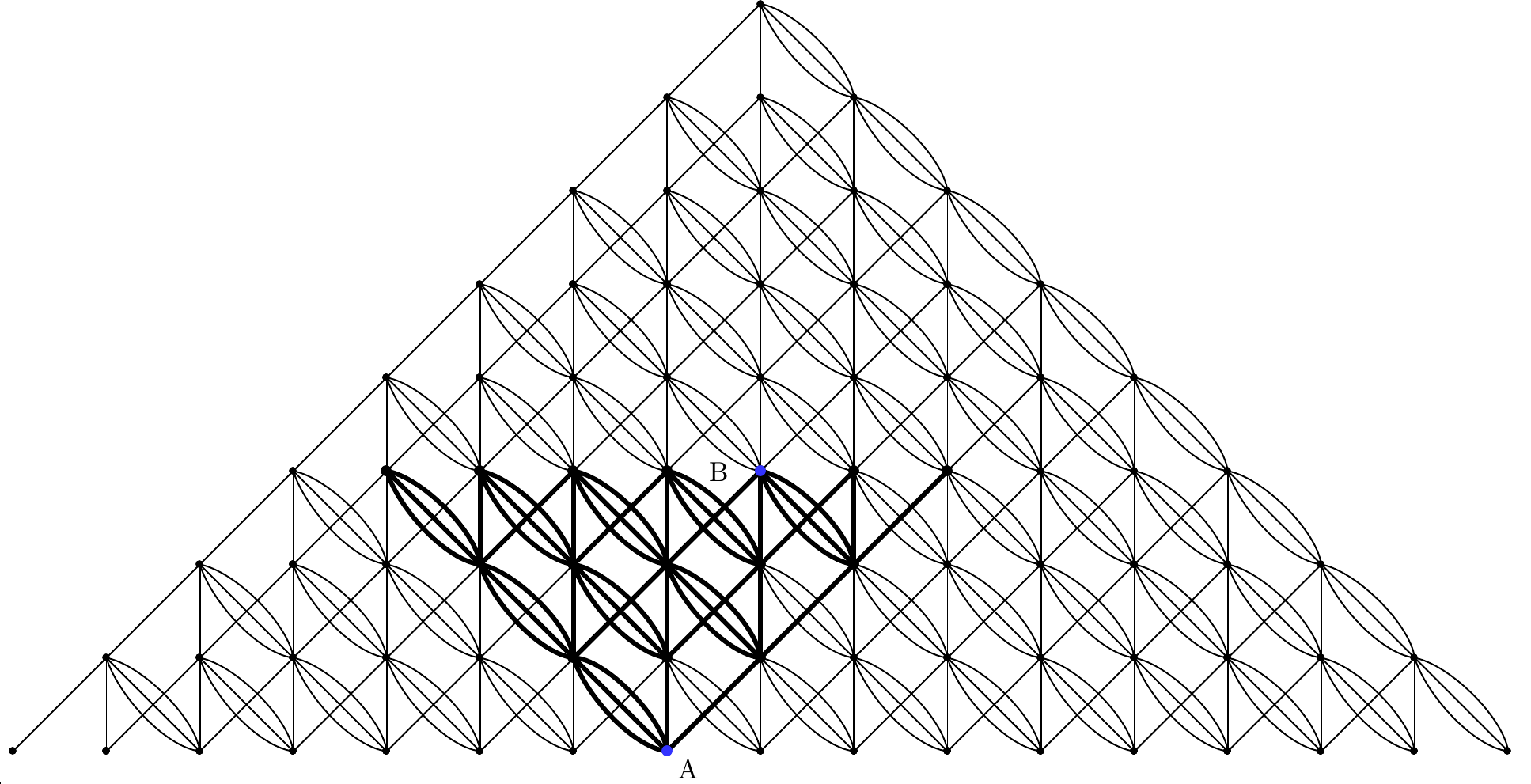}
                \caption{Vertices   $A$ and $B$ in the graph $B_p$. }
                \label{Fig:InvertedGraph}
\end{figure}

\begin{proof} We can assume that $\bar{n}>2n_1,$ where $n_1=n_1(a_0,a_1,\dots,a_d)$, is defined in the proof of Lemma \ref{Lm:genBinEstimates}.
 Let $l_0, 0\leqslant l_0\leqslant Nd,$ be such that coefficient $\alpha_{l_0}$ is nonzero. We can rewrite the right hand side of  \eqref{ineq:generalizedBinomialIneq} as follows:
\[R\gamma_{n,k}= C_d(n-N,k-l_{0})P(n,k,l_0), N\leqslant n\leqslant \bar{n}, 0\leqslant k\leqslant nd,\]
where $P(n,k,l_0)$ is defined by $\sum\limits_{l=0 }^{N}\alpha_{l}\, \frac{C_d(n-N,k-l)}{C_d(n-N,k-l_0)}$.
Let $\alpha$ denote the maximum of $|\alpha_l|, 0\leqslant l\leqslant Nd.$
We want to show that there is a polynomial $Q(x)$ of degree $\text{deg}(Q)\leq Nd$ such that
\begin{equation}\label{ineq:P-deviation}
|P(n,k,l_0) - P(\bar{n},\bar{k},l_0)|\leqslant Q(\bar{n}). \end{equation}
It is enough to show that there is $c_1>0$ such that $|\frac{C_d(n-N,k-l)}{C_d(n-N,k-l_0)}|\leqslant c_1 n^{Nd}, 0\leqslant l\leqslant Nd, N\leqslant n\leqslant \bar{n}$.
The latter inequality follows from $Nd$ fold  application of  part $1$ of Lemma \ref{Lm:genBinEstimates}.
Define function $\tilde{Q}$  by  $\tilde{Q} =P(n,k,l_0) - P(\bar{n},\bar{k},l_0)$. We can write

\begin{multline}\gamma_{n,k} =\frac{1}{R}C_d(n-N,k-l_{0})P(n,k,l_0) =\\ =\frac{C_d(n-N,k-l_{0})}{C_d(\bar{n}-N,\bar{k}-l_{0})}\frac{C_d(\bar{n}-N,\bar{k}-l_{0})P(n,k,l_0)}{R}  =\\=\frac{C_d(n-N,k-l_{0})}{C_d(\bar{n}-N,\bar{k}-l_{0})}\frac{C_d(\bar{n}-N,\bar{k}-l_{0})(P(\bar{n},\bar{k},l_0)+\tilde{Q})}{R}.  \end{multline}
By the assumption we have $|\gamma_{\bar{n},\bar{k}}|=|\frac{1}{R}P(\bar{n}, \bar{k},l_0)C_d(\bar{n}-N,\bar{k}-l_0)|\leqslant 2.$ Therefore inequality~\eqref{ineq:P-deviation} can be written as $|\tilde{Q}|\leqslant Q.$ We get
\[|\gamma_{n,k}|\leqslant3\frac{Q(\bar{n})C_d(n-N,k-l_{0})}{C_d(\bar{n}-N,\bar{k}-l_{0})}.\]
Applying the estimate from  part  $2$ of Lemma \ref{Lm:genBinEstimates}  $(\bar{n}-n)$ times and using assumptions on the vertices  $A$ and $B$,
we obtain that  $\frac{C_d(n-N,k-l_{0})}{C_d(\bar{n}-N,\bar{k}-l_{0})}\leqslant 3e^{-\tilde{C}(\delta)(\bar{n}-n)}$ for some $\tilde{C}(\delta)>0.$
Finally we  get (an independent of the initial choice of $l_0$) estimate:
\[|\gamma_{n,k}|\leqslant3\frac{Q(\bar{n})C_d(n-N,k-l_{0})}{C_d(\bar{n}-N,\bar{k}-l_{0})} \leqslant 3e^{-C(\delta)(\bar{n}-n)}\]
for some  $C(\delta)>0.$

\end{proof}

\subsection{Examples of  limiting curves}

Let $q_1$ and $q_2$ be two numbers (parameters) from $(0, 1).$
We consider the function $S^{p}_{q_1,q_2}:[0,1]\rightarrow [0,1]$ that maps a number $x$ with $q_1$-$r$-adic representation
$x = \sum\limits_{j=1}^\infty I_{q_1}(\omega_j)q_1^j\Big(\frac{t_{q_1}}{{q_1}}\Big)^{\bar{a}_1\cdot\bar{s}_j+2\bar{a}_2\cdot\bar{s}_j+\dots+d\,\bar{a}_d\cdot\bar{s}_j}$   to \begin{equation}
\label{eq:explicit_Sq1q2}
S^{p}_{q_1,q_2}(x) =  \sum\limits_{j=1}^\infty I_{q_2}(\omega_j)q_2^j\Big(\frac{t_{q_2}}{{q_2}}\Big)^{\bar{a}_1\cdot\bar{s}_j+2\bar{a}_2\cdot\bar{s}_j+\dots+d\,\bar{a}_d\cdot\bar{s}_j}.
\end{equation}
For any $q_1$-$r$-\emph{stationary} point $x_0=\sum\limits_{j=1}^m I_{q_1}(\omega_j)q_1^j\Big(\frac{t_{q_1}}{{q_1}}\Big)^{\bar{a}_1\cdot\bar{s}_j+2\bar{a}_2\cdot\bar{s}_j+\dots+d\,\bar{a}_d\cdot\bar{s}_j}$
and any $x \in [0, 1]$ the function $S^{p}_{q_1,q_2}$ satisfies the following self-affinity property:
\begin{equation}
\label{eq:selfSimularity_Sq1q2}
S^{p}_{q_1,q_2}\Big(x_0+r_{q_1}x\Big) =     S^{p}_{q_1,q_2}(x_0) + r_{q_2}S^{p}_{q_1,q_2}(x),
\end{equation}
where $r_{q_i} = q_i^m\Big(\frac{t_{q_i}}{{q_i}}\Big)^{\bar{a}_1\cdot\bar{s}_m+2\bar{a}_2\cdot\bar{s}_m+\dots+d\,\bar{a}_d\cdot\bar{s}_m}, i=1,2.$
Expression \eqref{eq:selfSimularity_Sq1q2} means that the graph of $S^{p}_{q_1,q_2}$ considered on the $q$-$r$-adic interval  $[x_0, x_0+r_{q_1}]$ coincides after  renormalization with the graph of $S^{p}_{q_1,q_2}$ on the whole interval $[0,1]$. Also for $q_1=1/r$ function $S^{p}_{1/r,q_2}$ is the distribution function of the measure~$\tilde{\mu}_{q_2}.$

Functions $S^{p}_{q_1,q_2}(\cdot)$ allow us to define new functions
\[\T_{p,q_1}^{k} : = \frac{\partial^k S^{p}_{q_1,q_2}}{\partial q_2^k}\Big|_{q_2=q_1},\ k\in\mathbb{N}.\]
If $k=0$ we will assume that $\T_{p,q}^{0} (x) =x.$
For $q={1/2}$ and  $k=1$ function $\frac12 \T_{1+x,{1/2}}^{1}$ is the Takagi function, see~\cite{Takagi1903}. The function $\T_{p,q_1}^{k} $ on the interval $[x_0, x_0+r_{q_1}]$  can be expressed by a linear combination of the functions $\T_{p,q_1}^{j},0\leqslant j\leqslant k.$ (Expression can be easily obtained  by differentiating  identity \eqref{eq:selfSimularity_Sq1q2} with respect to parameter $q_2$ and defining $q_2$ equal to $q_1$.)

\begin{figure}[t!]
                \centering
                \includegraphics[scale=0.6]{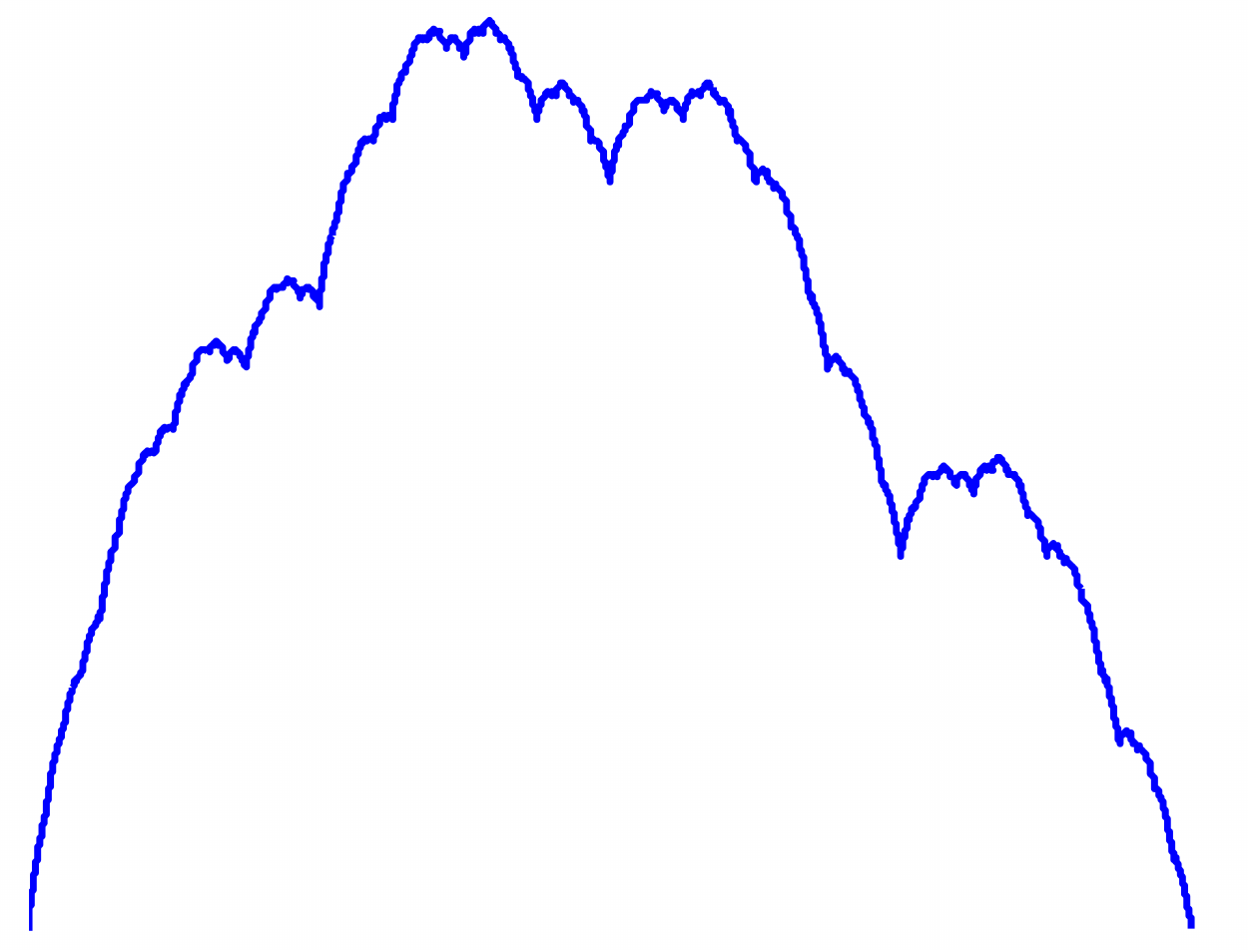}
                \caption{Graph of the funciton $\T_{p,q}^{1}$, defined by polynomial $p(x)=1+x+2x^2$ with parameter $q$ equal to $\frac14.$ }
                \label{Fig:systemSLP[1,1,2]}
\end{figure}

\begin{Theorem} Functions  $\T_{p,q}^{k}, q\in(0,1/a_0), k\geq1,$ are continuous functions  on $[0,1]$.
\label{TpqIsCont}
\end{Theorem}
\begin{proof}

The proof is based on the fact that  any two points $x$ and $y$ from the same  $q$-$r$-adic interval of rank $m$ have the same coordinates $(\omega_1, \omega_2,\dots, \omega_m)$ in $q$-$r$-adic expansion. This provides a straightforward estimate for the difference
$|\T_{p,q}^{1}(x) - \T_{p,q}^{1}(y)|$.

Let $b=b_q$ denote  the ratio $t_q/q$. As shown in Section $3.3$ above any  $x$ in $(0,1)$ can be coded by a  path $\omega=(\omega_i)_{i=1}^\infty, \omega_i\in\{0,1\dots,r-1\}=\mathcal{A},$  in $r$-adic (perfectly balanced) tree $\mathcal{M}_{r}$. The function
$\T_{p,q}^{1}$ maps $x=x_q\in[0,1]$ with $q$-$r$ adic series representation
 \[x=\sum\limits_{j=1}^\infty\Big(\sum\limits_{i=0}^{\omega_j-1}b^i\Big)q^jb^{\bar{a}_1\cdot\bar{s}_j+2\bar{a}_2\cdot\bar{s}_j+\dots+d\,\bar{a}_d\cdot\bar{s}_j},\]
 to $z=\frac{\partial}{\partial q} x_q. $
Let $\tilde{s}_j$ denote the sum $\bar{a}_1\cdot\bar{s}_j+2\bar{a}_2\cdot\bar{s}_j+\dots+d\,\bar{a}_d\cdot\bar{s}_j$.

Derivative $\frac{\partial}{\partial q}(q^jb^l)$ equals to $q^{j-1}b^{l-1}[(j-l)b+lt'_q],$
where $l =\tilde{s}_j-\omega_j+i.$
Using implicit function theorem we find that
\begin{equation}\label{eq:IFT}
t'_q = -\frac{a_0d q^{d-1}+a_1(d-1)q^{d-2}t_q+\dots+a_{d-1}t_q^{d-1}-(d-1)q^{d-2}}{a_1 q^{d-1}+2a_2q^{d-2}t_q+\dots a_dt_q^{d-1}d}.\end{equation}
Let also $\AAA$ denote the maximum of the coefficients  $\{a_0,\dots,a_d\}$ of the polynomial $p(x)$.
We have  $|t'_q|\leq \AAA\frac{2d^2}{q}.$ Let $p_{max}\in(0,1)$ denote the maximum of $ \{q,t_q,\frac{t_q^2}{q},\dots,\frac{t_q^d}{q^{d-1}}\}.$

Assume  $y$ is  the left  boundary of some  $q$-$r$-adic interval of rank $m,$ containing point $x$. Then the following inequality holds (we simply  write $T$ for $T_{p,q}^{1}$):
\[|T(y)-T(x)| \leq \sum\limits_{j=m}^\infty\sum\limits_{i=0}^{\omega_j-1}\big|\frac{\partial}{\partial q} \big( q^jb^{\tilde{s}_j+i} \big)\big|. \]

Using estimate $ | q^jb^{\tilde{s}_j+i} |\leq (p_{max})^j$, $0\leq i\leq r-1$, we see that the absolute value of $\frac{\partial}{\partial q}(q^jb^l)$ for $j>2$ is estimated by expression $P(j,q)(p_{max})^{j-2}$, where $P(j,q)$  is some polynomial. Define $\varepsilon$ to be equal to $0.99.$ Then for  $m$ large enough it holds:

\begin{equation}|T(y)-T(x)|\leq \sum\limits_{j=m}^\infty\sum\limits_{i=0}^{\omega_j-1}P(j,q)(p_{max})^{j-2} \leq C(p_{max})^{m\varepsilon},
\label{eq:increamentEstimate}\end{equation}
where $C$ is some constant.

In general case we can assume that points $x$ and $x+\delta$ are from some $q$-$r$-adic interval of rank $m=m(\delta), \lim\limits_{\delta\rightarrow 0}m(\delta)=+\infty, $ and let  $y$ be the  left boundary point of this interval. Then
\[|T(x+\delta)-T(x)| \leq |T(y)-T(x)|+|T(y)-T(x+\delta)|\leq 2Cp_{max}^{m\varepsilon}\]
For $k>1$ we  can use  a similar argument based on the following  estimate for the $k$-th derivative: $|\frac{\partial^k}{\partial q^k}q^jb^l| \leq P_k(j,q)(p_{max})^{j-k-1},$ where $j>k$ and $P_k(j,q)$ is some polynomial.
\end{proof}

\begin{proposition} 
For a cylindrical function $g=- \sum_{j=0}^{d}ja_j\mathbbm{1}_{\{k^1(x_1)=j\}}\in\mathscr{F}_1$ and  for  $\mu_{{q}}$-a.e. $x$ there is a stabilizing sequence $l_n(x)$ such that  the limiting function is $ \T_{p,q}^{1}$.
\end{proposition}

\begin{proof}

For simplicity we will present the proof for $p(x)=1+x+x^2.$
The general case  follows the same steps. Theorem \ref{Th:ExistanceOfLimitShapeSLP} implies that we can find the limiting function $\varphi(x)$ as $\lim\limits_{n\rightarrow\infty} \varphi_{n,k}$, where (by the law of large numbers)  $\frac{k_n}{n}\rightarrow \mathbb{E}_{\mu_q}k_1$. Lemma \ref{Lm:triangArray} imply that it is sufficient to show that the function $\varphi(x)$ coincide with $\T_{p,q}^{1}$ at $x=q^j$ and $x=q^{j-1}(q+t_q)$, where $j\in\mathbb{N}$.

The function $\T_{p,q_1}^{1}$ maps point $x =q^j$ to  $\frac{\partial }{\partial q}q^j =jq^{j-1}$  and point $x=q^{j-1}(q+t_q)$ to $q^{j-2}\big(jq+(j-1)t_q+t'_qq\big)$.
Using expression \eqref{eq:IFT} we see that
 $t'_q =\frac{1-(2q+t_q)}{2t_q+q}.$

Identity \eqref{eq:GenBinIdentity} implies that  $h_{n,k}^g=  \frac{k}{n}H_{n,k}.$ We need to find the following limits for  $i\in\mathbb{N},$ $n\rightarrow \infty$ and $\frac{k_n}{n}\rightarrow \mathbb{E}_{\mu_q}k_1 = 2q+t_q$ (we write $F$ for $F_{n,k}$):

\begin{enumerate}
\item $\lim\frac{1}{R_n}\Big(F(H_{n-i,k-2i}) - \frac{H_{n-i,k-2i}}{H_{n,k}}F(H_{n,k})\Big)$
\item $\lim\frac{1}{R_n}\Big(F(H_{n-i,k-2i}+H_{n-i,k-2i+1}) - \frac{H_{n-i,k-2i}+H_{n-i,k-2i+1}}{H_{n,k}}F(H_{n,k})\Big)$
\end{enumerate}
We define the normalizing coefficient $R_n$ by
$R_{n}=\frac{qH_{n,k}}{n}(2-\mathbb{E}_{\mu_q}k_1)$. After some computations we see that the first limit equals $iq^{i-1},$ and the second to $q^{i-2}\big(iq+(i-1)t_q+t'_qq\big)$. These shows that that the limiting function $\varphi$ coincides with the function $\T_{p,q_1}^{1}$ on a dense set of $q$-$2$-stationary points.  Therefore, by Theorem \ref{TpqIsCont} these functions coincide.

\end{proof}

Numerical simulations show that limiting functions  $\T_{p,q}^{k}, k\geqslant1,$  and their linear combinations arise as limiting functions
 $\lim\limits_{n\rightarrow \infty}\varphi_{n,k_n}^g$ for a general cylindrical function $g\in\mathcal{F}_N$.
We do not have any proof of this statement except for the case of the Pascal adic, see Theorem \ref{Thm:TakagiOnly} above. Expression \eqref{eq:Numformula} shows  that for a cylindrical function  $g\in\mathscr{F}_N$ the partial sum $F^g_{n,k}$  is defined by the coefficients $h_{N,k}^g, 0\leq k\leq Nd.$ Its seems to be useful to define  $h_{N,k}^{g_m}$ by the generating function $h_{N,k}^{g_m}=\text{coeff}[v^m]\, (h_0+h_1v+\dots+h_{d}v^{d})^{k}p(v)^{Nd-k}, $ where functions $g_m$ forms an orthogonal basis. (For the Pascal adic the function $(1-av)^{k}(1+v)^{n-k}, a = \frac{1-q}{q},$ is the generating function of the Krawtchouk polynomials and the basis $g_m$ is the basis of Walsh functions, see \cite{LodMin15}).

\section{Limit of limiting curves}

In this section we answer the question by  \'E. Janvresse, T. de la Rue and Y.~Velenik from \cite{DeLaRue}, page 20, Section 4.3.1.

 \begin{figure}[h!]
                \centering
                \includegraphics[scale=0.27]{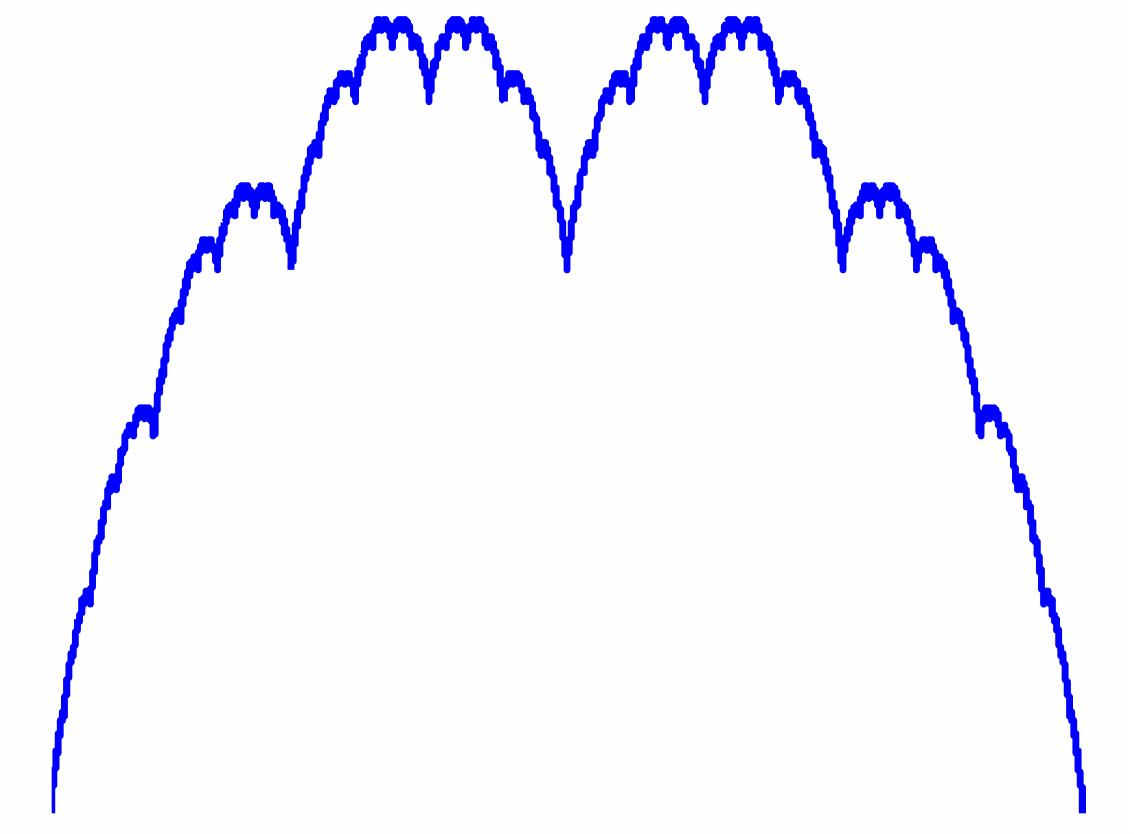}
                \includegraphics[scale=0.27]{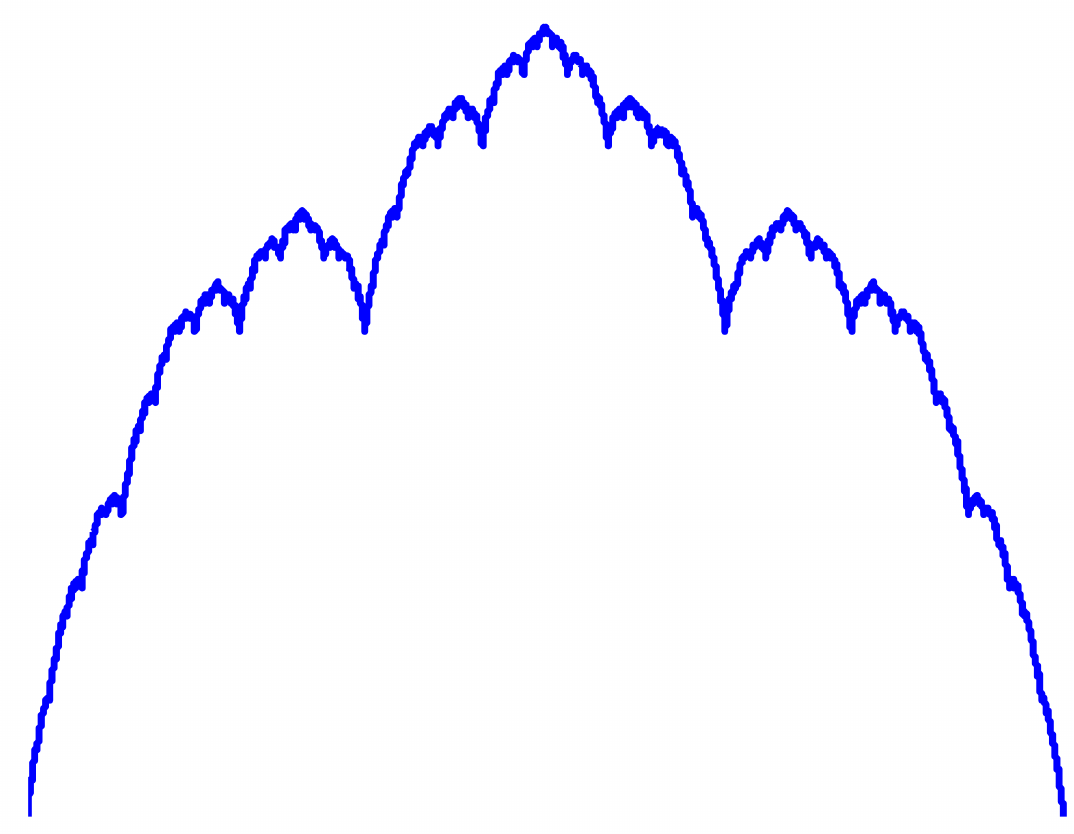}
                 \includegraphics[scale=0.27]{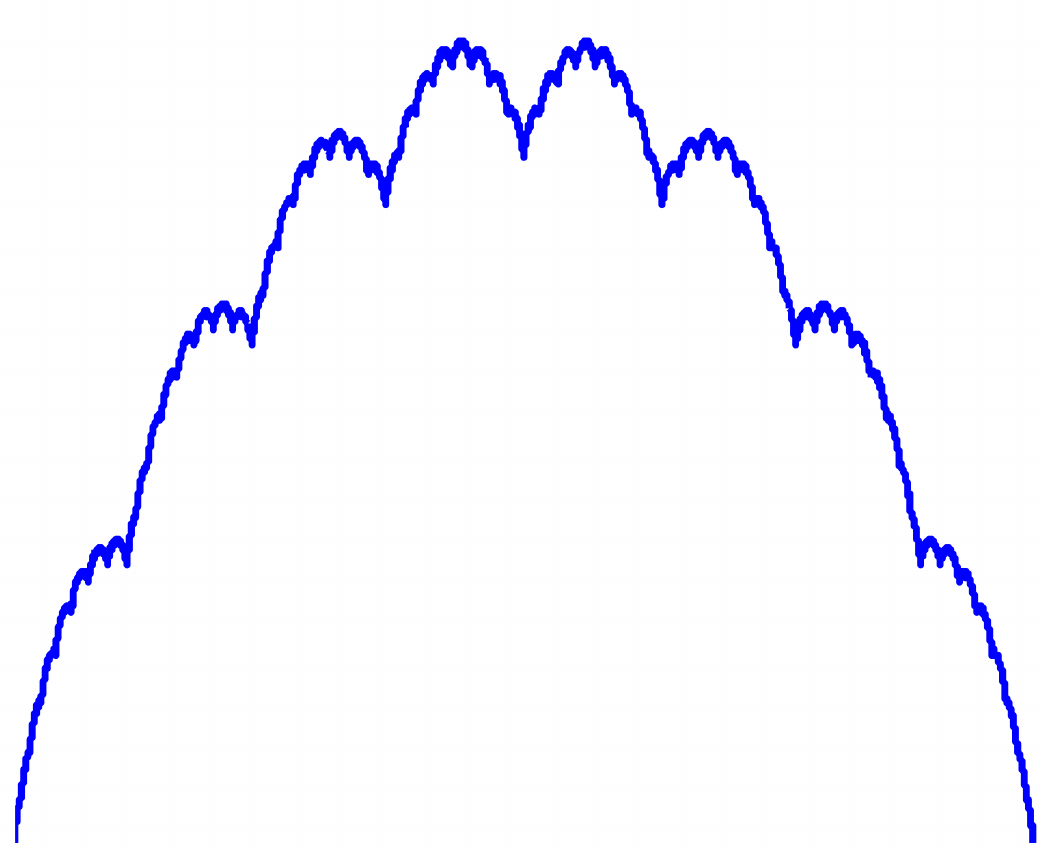}
                 \includegraphics[scale=0.27]{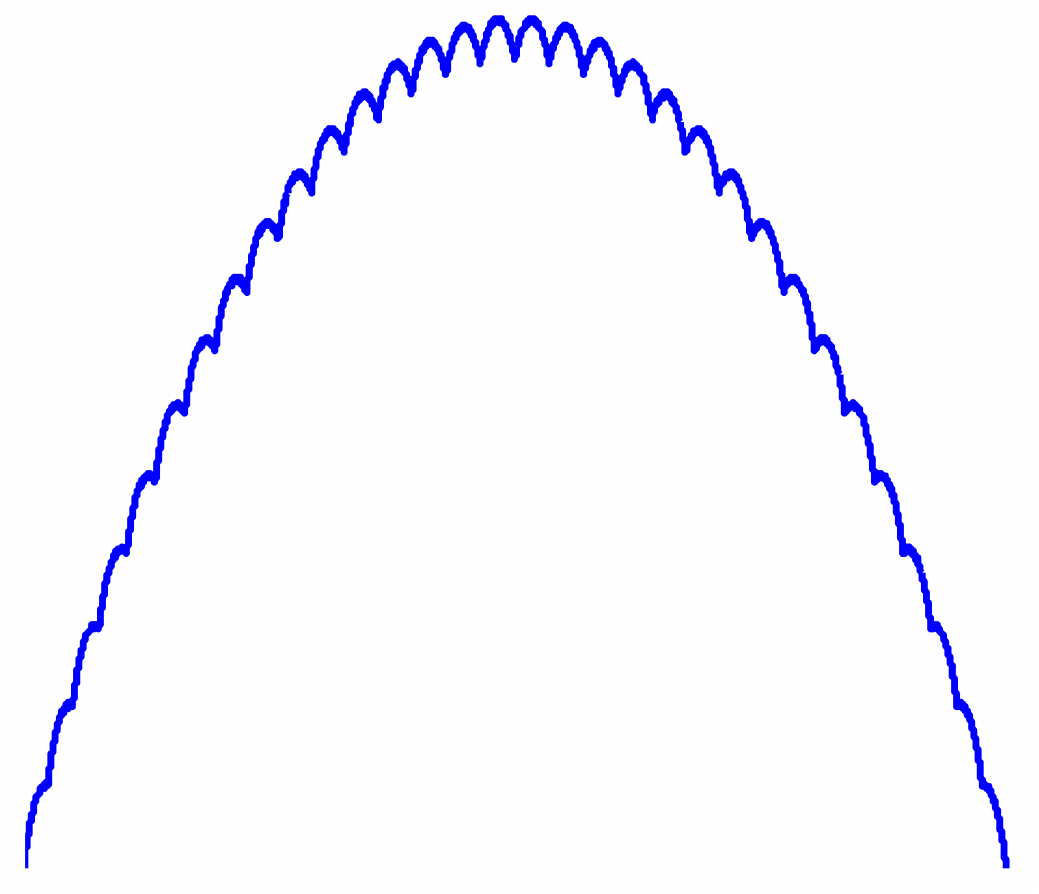}
                \caption{Limiting curves observed for the polynomial adic transformations (from left to right): $d+1=2,3,8,32.$ }
                \label{Fig:somecurves}
\end{figure}

Let  $q\in (0,1)$ and $t_q\in(0,1)$ be the unique solution in $(0,1)$ of the equation
\[q^d+q^{d-1}t+\dots+t^d=q^{d-1}.\]

As above, we denote by $b=b_q$ the ratio $t_q/q$. Any  $x$ in $(0,1)$ has an almost unique $(d+1)$-adic representation:
\begin{equation}x = \sum\limits_{j=1}^\infty\Big(\sum\limits_{i=0}^{\omega_j-1}b^i\Big)q^jb^{s_j^1+2s_j^2+\dots ds_j^d-\omega_j},\label{eq:symmetricNadic}\end{equation}
where $\omega=(\omega_i)_{i=1}^\infty, \omega_i\in\{0,1\dots,d\}=\mathcal{A},$ is a path in $(d+1)$-adic (perfectly balanced) tree $\mathcal{M}_{d+1}$ and $s_j^k$ is the number of occurrences of letter $k$ among $(\omega_1,\omega_2, \dots, \omega_j).$

We denote by $S_q(x)$ the  (anaclitic in parameter $q$) function defined by (a uniformly summable in $x$) series \eqref{eq:symmetricNadic}. We put $q_*$ equal to $1/(d+1)$  (this is so called symmetric case $t_{q_*}=q_*$). If $d=1$  representation \eqref{eq:symmetricNadic} for  $q^*=1/2$ is a usual dyadic representation of $x\in(0,1)$.

The authors of  \cite{DeLaRue} were interested in the limiting behavior of the graph of the function   \[\tilde{\T}_{d}: x \mapsto \frac{1}{d+1}\frac{\partial S_{q}}{\partial q}\Big|_{q=q_*}\] for large values of  $d$ (we also introduced vertical normalization by $d+1$, if  $d=1$ the graph of $\frac12T_{1}$ is the Takagi curve). On the basis of a series of numerical simulations they noticed that limiting curves for $d\rightarrow\infty$ seem to converge to a smooth curve. Below we will  show that the limiting curve for $d=\infty$ is actually a parabola, see Fig. \ref{Fig:somecurves}.

We are going split the unit interval into $d+1$   subintervals $I_i=(\frac{i}{d+1};\frac{i+1}{d+1}),0\leq i\leq d$, of equal length and evaluate the function  $\tilde{\T}_{d}$ at each of the (left) boundary points of these intervals. 
We also want to show that the function $\tilde{\T}_{d}$ is uniformly in $d$ bounded at these intervals. After that we  go to the limit in~$d$.

%

Symmetry assumption  $q = t_q = q_*$ and implicit function theorem (see \eqref{eq:IFT}) imply that
$t'_{q_*} = - \frac{2-d}{d}.$
In its turn this implies $b'_q = \frac{t'}{q}-\frac{b_q}{q}\Big|_{q=q_*} = -\frac{2(d+1)}{d}.$
Finally we find that $\frac{\partial }{\partial q} (q^jb^r)  = jq^{j-1}b^r+rb^{r-1}q^jb'_q \Big|_{q=q_*} =(q_*)^{j-1}(j+rq_*(-\frac{2(d+1)}{d}))= (q_*)^{j-1}(j - \frac{2r}{d}))$.


Note that the left boundary point $a_d$ of $I_{ad}, a\in[0,1],ad \equiv [ad], $ ($[\,\cdot\, ]$ is an integer part) equals  $a_d=\frac{ad}{d+1}$ and is coded by the stationary path  $\omega = (\omega_j)_{j=1}^\infty\in \mathcal{M}_{d+1}$ with $ \omega_1 = ad$ and $ \omega_j \equiv 0, j\geq2.$

We have
\[\tilde{\T}_{d}(a_d) = \frac{1}{d+1}\sum\limits_{i=0}^{da-1}\Big(j-\frac{2(da(j-1)+i)}{d}\Big) =a_d(1-a_d)\frac{d+1}{d} \xrightarrow[d\rightarrow\infty]{} a_d(1-a_d). \]

This shows that the smooth curve (if exists) should be a parabola.

To complete the proof of the theorem it only remains to show that $\tilde{\T}_{d}(x)$ is uniformly bounded in $d$ at the intervals $I_{ad}$. Analogously to \eqref{eq:increamentEstimate} we see that for $x\in I_{ad}$ it holds

$|\tilde{\T}_{d}(x)-\tilde{\T}_{d}(a_d)| = \frac{1}{d+1}\sum\limits_{j=2}^\infty q_*^{j-1}\sum\limits_{i=0}^{\omega_j-1}\Big(j-\frac{2(s_j^1+2s_j^2+\dots + ds_j^d-\omega_j+i)}{d}\Big)\leq \frac{100 d}{d+1} \sum\limits_{j=2}^\infty j q_*^{j-1} \leq \\ \leq\frac{400d}{(d+1)^2}.$ 
That finishes our proof.


\subsection{Question.}

We may heuristically interpret results of Section $4$ as existence of a limiting curve of a dynamical system defined by a diagram  with "infinite"\, number of edges.
This leads us to the following questions: Does this system really exist? How to define it correctly? Which properties does it have?







\end{document}